\documentclass[11pt]{article}%
\usepackage{amsmath}
\usepackage{amsfonts}
\usepackage{amssymb}
\usepackage{graphicx}

\usepackage[all]{xy}

\providecommand{\U}[1]{\protect\rule{.1in}{.1in}}
\newcommand{\ba}{\begin{eqnarray}}
\newcommand{\ea}{\end{eqnarray}}
\newcommand{\bas}{\begin{eqnarray*}}
\newcommand{\eas}{\end{eqnarray*}}
\newcommand{\be}{\begin{equation}}
\newcommand{\ee}{\end{equation}}

\newcommand{\N}{{\mathbb N}}
\newcommand{\C}{{\mathbb C}}

\newcommand{\R}{{\mathbb R}}
\newcommand{\Z}{{\mathbb Z}}

\newcommand{\kc}{{\cal K}}

\newcommand{\hc}{{\mathcal H}}
\newcommand{\gc}{{\mathcal G}}

\def\hoj{\hc(\omega, J_0)}
\def\goj{\gc(\omega, J_0)}
\def\gco{G_\C}

\newcommand{\hm}{\hc am}
\newcommand{\hmc}{\hc am_\C}
\newcommand{\om}{\omega}

\newtheorem{theorem}{Theorem}
\newtheorem{proposition}[theorem]{Proposition}
\newtheorem{corollary}[theorem]{Corollary}
\newtheorem{lemma}[theorem]{Lemma}
\newtheorem{example}[theorem]{Example}
\newtheorem{definition}[theorem]{Definition}
\newenvironment{proof}[1][Proof]{\noindent\textbf{#1.} }{\ \rule{0.5em}{0.5em}}
\newtheorem{preremark}[theorem]{Remark}
\newenvironment{remark}{\begin{preremark}\rm}{\hfill$\Diamond$\end{preremark}}
\newtheorem{prenotation}[theorem]{Notation}

\numberwithin{equation}{section}
\numberwithin{theorem}{section}
\begin{document}

\title{{On complexified analytic Hamiltonian flows and geodesics on the space of K\"ahler metrics}}
%\title{{On geodesics past the boundary of the space of K\"ahler metrics}}
\author{Jos\'{e} M. Mour\~{a}o\footnote{jmourao@math.tecnico.ulisboa.pt} and Jo\~{a}o P. Nunes\footnote{jpnunes@math.tecnico.ulisboa.pt, corresponding author}}
\maketitle

\begin{center}
Center for Mathematical Analysis, \\ Geometry and Dynamical Systems\\and\\ 
Department of Mathematics\\ Instituto Superior T\'ecnico\\ Universidade de Lisboa\\Av. Rovisco Pais\\ 1049-001 Lisbon, Portugal\\
\end{center}

\bigskip

\begin{abstract}
For a compact real analytic symplectic manifold $(M,\omega)$ we describe 
an approach to the complexification of Hamiltonian flows \cite{Se,Do1,Th1} 
and corresponding geodesics on 
the space of K\"ahler metrics. In this approach, motivated by recent work on 
quantization, the complexified Hamiltonian flows act, through the Gr\"obner theory of Lie series, 
on the sheaf 
of complex valued real analytic functions, changing the sheaves 
of holomorphic functions. This defines an action on the space
of (equivalent) complex structures on $M$ and also a direct action 
on $M$. This description is related to the approach of \cite{BLU}  where one has an action on a complexification $M_\C$ 
of $M$ followed by projection to $M$.   
Our approach allows for the study of some Hamiltonian functions which are not real 
analytic. It also leads naturally to the consideration of continuous degenerations of 
diffeomorphisms and of K\"ahler structures of $M$.
Hence, one can link continuously (geometric quantization) real, and more 
general non-K\"ahler, polarizations with K\"ahler polarizations. 
This corresponds to the extension of the geodesics to the boundary of the space of K\"ahler 
metrics. Three  illustrative examples are considered.

We find an explicit formula
for the complex time evolution of the K\"ahler potential under the flow. 
For integral symplectic forms, this formula corresponds to the complexification of 
the prequantization of Hamiltonian symplectomorphisms.
We verify that certain families 
of K\"ahler structures,
which have been studied in geometric quantization, are geodesic families. 
\end{abstract}
\tableofcontents

%\newpage

\section{Introduction}
\label{s1}

The formal complexification $\hc am_\C(M,\omega)$ of the group of Hamiltonian diffeomorphisms
$\hc am (M,\omega)$ of a symplectic manifold $(M, \omega)$ has been studied recently in
two main contexts.

In K\"ahler geometry, the observation that the space, $\kc_{[\omega]}$, of K\"ahler metrics with
a fixed cohomology class  has the
structure of an infinite-dimensional symmetric space of the form
$\hmc(M,\omega)/ \hm(M,\omega)$, led Semmes \cite{Se} and Donaldson \cite{Do1, Do2} 
to the result that the Hamiltonian flows in imaginary time lead to
geodesics in $\kc_{[\om]}$ with respect to the Donaldson-Semmes-Mabuchi metric \cite{M}.
These geodesics play an important role in the recent breakthrough
in K\"ahler geometry, completing the proof of the link of
$K$-stability of a Fano variety with the existence of a
K\"ahler-Einstein metric \cite{Ti,St,Be,CDS}. (See \cite{PS} for a review.) 

The second context in which $\hmc(M,\omega)$ has been studied has been in 
quantization. 
The applications of imaginary time Hamiltonian flows to quantization  
originated in the work of Thiemann \cite{Th1,Th2} in the context on 
nonperturbative quantum gravity. There, a complex canonical
transformation is used to take the gravitational spin $SU(2)$-connection to the complex Ashtekar 
$SL(2,\C)$ connection. 

In the context of geometric quantization, $\hmc(M,\omega)$ acts
naturally in the space of (geometric quantization) polarizations
by acting on the local functions defining the polarization. This idea has been 
adapted to the framework of geometric quantization
by Hall and Kirwin in \cite{HK1}. It has been further
investigated, for complexifications of Lie groups of compact type, 
in \cite{KMN1,KMN2}. Previous work where these ideas were also present include
\cite{FMMN1,FMMN2}, \cite{BFMN} for symplectic toric manifolds and 
\cite{LS12,Lempert-Szoke10}.  In a series of papers \cite{RZ1,RZ2,RZ3}, Rubinstein and Zelditch have also 
related ideas of quantization to the study of geodesics on the space of K\"ahler metrics.

For geometric quantization the study of complex time Hamiltonian flows has
the advantage of linking continuously the 
(mathematically easier to define) quantum Hilbert spaces for
K\"ahler polarizations with the more difficult to define, especially in
the case of singular polarizations, quantum Hilbert
spaces for real and mixed polarizations, \cite{HK1, KW1, KW2, KMN1, KMN2}.  One motivation for the present 
work was the verification that the K\"ahler families considered in the geometric quantization of cotangent bundles 
of Lie groups, abelian varieties and toric varieties correspond in fact to geodesic families.

As we show below in examples, the imaginary time Hamiltonian flows in geometric quantization 
extends the flows considered in
K\"ahler geometry past the boundary of the space of K\"ahler polarizations. (For geodesics hitting the boundary 
of the space of K\"ahler metrics see, for example, \cite{RZ2}.) While the ``metrics'' 
beyond the boundary will be degenerate or become indefinite on subsets, the corresponding 
polarizations can still be interesting from a geometric quantization point of view.

In \cite{BLU}, the authors also study properties of complexified Hamiltonian flows in  
real analytic manifold $(M,\omega)$. They describe the one-parameter families of diffeomorphisms given by 
the complexification of (real-time) Hamiltonian flows in $(M,\omega)$ 
in terms of flows on a complexification $(M_\C,\omega_\C)$, which is a holomorphic symplectic manifold. 
These one-parameter families of diffeomorphisms on $(M,\omega)$ are 
then obtained by projection along leaves of a time evolving 
holomorphic foliation of $(M_\C,\omega_\C)$.  The projection of the flow does not, 
in general, lead to a flow on $(M,\omega)$.

In the present work, motivated by \cite{HK1,KMN1}, we give a complementary approach 
for real analytic $(M,\omega)$, 
using the Gr\"obner formalism of Lie series \cite{G1,G2}. Let $C^{\rm an}(M)$ be the space of complex 
valued real analytic functions on $M$. This formalism defines 
the exponential $e^{tX_h}$, for the Hamiltonian flow of (real)
$h\in C^{\rm an}(M)$ acting (locally) on $C^{\rm an}(M)$, 
as a series. This point of view in particular, allows for 
explicit formulas for the geodesics and for 
the complex time evolution of the K\"ahler potential. In this formalism, as we show in a few examples, 
by letting the complexified Hamiltonian flow go to
the boundary of the space of K\"ahler metrics, one can study real and mixed polarizations 
(in the sense of geometric quantization). 
The close link to 
geometric quantization also 
clarifies the series expression for the geodesics in the space of K\"ahler potentials, as explained in Section \ref{gqsection}.
When it hits this boundary, the complexified Hamiltonian flow can still continue to exist 
on the complexification $M_\C$ of $M$, but the projection to $M$ ceases to be holomorphic. 
Notice also that, in our approach, we can also {\it in some cases} weaken the requirement of 
analiticity of $h$ and of $M$ and 
consider Hamiltonian flows of non real analytic Hamiltonian functions 
$h\in C^\infty(M)$. (See Section \ref{ssnonanal}.) 

Further applications to K\"ahler geometry, geometric tropicalization and quantization will be studied in \cite{MN,KMN3}.

\section{Complexification of analytic flows and action on complex structures}

We will use the formalism of {\it Lie series} introduced by Gr\"obner \cite{G1,G2}.
Let $M$ be a compact real analytic manifold and $X$ be a real analytic vector field on $M$ with analytic 
flow $\varphi_t$. 
Let $C^{\rm an}(M)$ denote the algebra of complex valued real analytic functions on $M$.
If $f\in C^{\rm an} (M)$ then
$f\circ \varphi_t \in C^{\rm an} (M), \forall t\in \R.$
Moreover, as shown in \cite{G1,G2}, and in particular from Theorem 3 of \cite{G2}, if follows that 
locally on any sufficiently small 
compact domain $K\subset M$, there exists a $T_{f,K}>0$ such that 
\begin{equation}\label{lieseries}
e^{tX}\cdot f := \sum_{k=0}^\infty \frac{X^k(f)}{k!} \, t^k,
\end{equation}
is an absolutely and uniformly convergent series for $|t|<T_{f,K}$, and defines a holomorphic 
function on $K\times \{t\in \C: |t|<T_{f,K}\}$. Following \cite{G2} we call the series in 
(\ref{lieseries}) a Lie series. Consequently, 

\begin{lemma}\label{grobner}
Let $M$ be a compact real analytic manifold and $X$ a real analytic vector field on $M$. For each $f\in C^{\rm an}(M)$ 
there exists $T_{f}>0$ such that for $\tau\in D_0=\{\tau\in \C: |\tau|<T_{f}\}$ the Lie series 
$$
\sum_{k=0}^\infty \frac{X^k(f)}{k!} \, \tau^k 
$$ 
defines a real analytic function on $M\times D_0$.
\end{lemma}

\begin{definition}\label{canbe}
We will say that {\it $e^{\tau X}$ can be applied to $f$} whenever there exists a $T_f>0$ such that the Lie series
$$
e^{\tau X}\cdot f:= \sum_{k=0}^\infty \frac{X^k(f)}{k!}\, \tau^k,
$$
is absolutely and uniformly convergent on compact subsets in $M\times D_0$, where $D_0=\{\tau\in \C: |\tau|<T_{f}\}$. 
\end{definition}

\begin{remark}\label{autom}
The operator $e^{\tau X}$ (see Theorem 5 in \cite{G2}) acts as a (local) automorphism of the algebra 
$C^{\rm an}(M)$, that is, for  $f,g\in C^{\rm an}(M)$ and $|\tau|$ sufficiently small,   
$$e^{\tau X}\cdot (fg) = (e^{\tau X}\cdot f)(e^{\tau X}\cdot g).$$ 
\end{remark} 

Moreover, as a direct consequence of Theorem 6 of \cite{G2},

\begin{theorem}\label{grobnerrr}
Let $D\subset \R^n$ be an open disk centered at the origin, $X$ a real analytic vector field on $D$ and $f\in C^{\rm an}(D)$ 
such that the power series
$$
f (x^1,\dots,x^n) = \sum_{I\in {\N}^n_0} a_Ix^I, \,\,\,\, I=(i^1,\dots,i^n),
$$ 
where $x^I= (x^1)^{i_1}\dots(x^n)^{i_n}$, converges in $D$. Assume that there exists $T>0$ such that, 
for $t\in\R$ with $|t|<T$, $e^{t X}$ 
can be applied to $f,x^1,\dots,x^n$ in a neighbourhood of $0\in D$. Then, one has, for $|t|$ sufficiently small,
$$
e^{t X}\cdot f (x^1,\dots,x^n) = \sum_{k=0}^\infty \frac{X^k(f)}{k!} t^k (x^1,\dots,x^n)=f (e^{t X} x^1,\dots,e^{t X} x^n)= 
f\circ \varphi_t (x^1,\dots,x^n)$$
in a neighbourhood of $0\in D$, where all Lie series above are absolutely and uniformly 
convergent on compact subsets.
\end{theorem}

Let now $\dim M=2n$ and let $J_0$ be a complex structure on $M$. Let $\{z^i\}_{i=1,\dots ,n}$ be a local system of $J_0$-holomorphic 
coordinates in an open neighbourhood $U$ of $p\in M$. 

\begin{theorem}\label{lemmalocal1}  
There exists $T>0$ such that for every 
$\tau\in D_{T}=\{\tau\in \C: |\tau|<T\}$ the functions
$$
z^i_\tau = e^{\tau X}\cdot z^i, i=1,\dots,n,
$$
form a system of complex coordinates on some open neighbourhood ${V\subset U}$ of $p$, defining a new 
complex structure $J_\tau$ on $V$ for which the coordinates $\{z^i_\tau\}_{i=1,\dots,n}$ are holomorphic.
\end{theorem}

\begin{proof} Let us consider an open neighbourhood of $p$, $V'\subset U$, with compact support $\overline{V'}\subset U$. 
The existence of $T$ such that the funcions $z^i_\tau$ are well defined on $V'$ for $|\tau|< T$ follows from Lemma \ref{grobner}. 
By taking a smaller $T$ if necessary, by continuity in $\tau$ and the fact that 
$dz^1\wedge \cdots \wedge dz^n \wedge d\bar z^1 \wedge \cdots \wedge d\bar z^n\neq 0$ at every point in $V'$,
we get that $dz^1_\tau \wedge \cdots \wedge dz_\tau^n \wedge d\bar z^1_\tau \wedge \cdots \wedge d\bar z_\tau^n\neq 0$ at 
every point in $V'$, for $|\tau| < T$. By the inverse function theorem, we obtain that $\{z^i_\tau\}_{i=1,\dots, n}$ form a
 new system of complex coordintes (that is, the real and imaginary parts form a system of real analytic coordinates) 
on some, possibly smaller, open neighbourhood $V\subset V'$ of $p$. 
\end{proof}

\begin{theorem}\label{ajax} There exists $T>0$ such 
that for $\tau\in D_{T}=\{\tau\in \C: |\tau|<T\}$ there exist a global 
complex structure $J_\tau$ on $M$, extending the complex structure
given in local $J_0$--holomorphic charts by Theorem \ref{lemmalocal1},
and a unique biholomorphism
\begin{equation}\label{mapa} 
\varphi_\tau: (M,J_\tau)\to (M,J_0) , 
\end{equation}
which, on local $J_0$-holomorphic coordinates,
acts as $e^{\tau X}$. 
\end{theorem}

\begin{proof} 
Since $M$ is compact, we can cover it with a finite atlas 
$\{(U_\alpha,z_{\alpha})\}_{\alpha=1,\dots,N}$ 
of $J_0$-holomorphic local charts.
Let $T_{\alpha, p}>0$, $V_{\alpha, p}$ be open sets, $p \in V_{\alpha, p} \subset U_\alpha$
such that the functions
$$
z_{\alpha, p, \tau} = e^{\tau X} \, z_{\alpha, p} =\left(e^{\tau X}\cdot z_{\alpha, p}^1, \dots, e^{\tau X}\cdot z_{\alpha, p}^n\right),
$$
where $z_{\alpha, p} = (z_{\alpha})|_{{V_{\alpha, p}}}$,
are defined on $V_{\alpha, p}$, are holomorphic functions of $\tau$
 and
$$
z_{\alpha, p, \tau}(V_{\alpha, p}) \subset z_\alpha(U_\alpha),
$$ 
for $\tau \in D_{T_{\alpha,p}}$. Let
$\{V_{\alpha_j, p_j}\}_{j=1, \dots, K}$ be a finite subcover of 
(the infinite open cover) $\{V_{\alpha, p}\}_{\alpha = 1, \dots, N, \, p \in M}$
and $\tilde T=\min_j \{T_{\alpha_j, p_j}\}$. Let $\phi_{\alpha_j \alpha_k}$ be the coordinate transition functions. 
From Theorem
\ref{grobnerrr} it follows that there exists $T, \, 0<T \leq \tilde T$ 
such that
\be
\label{tiri}
z_{\alpha_j, p_j, \tau} = e^{\tau X} \cdot z_{\alpha_j, p_j} 
=  e^{\tau X} \cdot \left(\phi_{\alpha_j \alpha_k} \circ  z_{\alpha_k, p_k}\right)
=  \phi_{\alpha_j \alpha_k} \circ z_{\alpha_k, p_k, \tau} \quad \forall \tau\in D_T.
\ee
We have therefore defined a new atlas 
$$\{(V_{\alpha_j,p_j},z_{\alpha_j,p_j, \tau})\}_{j=1,\dots,K}$$ on $M$,
with the same transition functions 
 restricted to the smaller open sets $V_{\alpha_j,p_j}\cap V_{\alpha_k,p_k}$.
Therefore the atlas $\{(V_{\alpha_j, p_j},z_{\alpha_j,p_j, \tau})\}_{j=1,\dots,K}$ defines a new 
complex structure $J_\tau$ on $M$, equivalent to $J_0$. For $\alpha=1,\dots,N$, denote by $\phi_{\alpha}: z_{\alpha}(U_{\alpha})\subset \C^n
\to U_\alpha$ the inverse of the coordinate function $z_{\alpha}$.
For $\tau\in D_T$, define the maps 
$$\varphi_{\tau,j}= \phi_{\alpha_j} \circ z_{\alpha_j, p_j, \tau}: V_{\alpha_j, p_j} \to U_\alpha.$$
We have, from (\ref{tiri}), that the maps $\{\varphi_{\tau, j}\}_{j=1,\dots,K}$ glue together to give a well defined global 
bijective map $\varphi_\tau: M \to M$. 
Indeed, surjectivity follows because $\varphi_\tau$ is a local diffeomorphism near every point and is homotopic to the 
identity so that it maps each connected component of $M$ onto itself. On the other hand, the inverse map $\varphi_{-\tau}$
exists, such that, on the charts $\varphi_\tau(V_{\alpha_j,p_j})$ one has
$$
e^{-\tau X} \cdot z_{\alpha_j,p_j,\tau} = z_{\alpha_j,p_j}.
$$ 
It is clear that $\varphi_\tau$ gives the unique biholomorphism from 
$(M,J_\tau)$ to $(M,J_0)$, such that
$z_{\alpha_j, p_j, \tau} = z_{\alpha_j, p_j} \circ \varphi_\tau$,
which proves the theorem. 
\end{proof}

\begin{remark}
Note that $J_\tau=\varphi_{\tau *}^{-1} J_0 := \varphi_{\tau *}^{-1}\circ J_0 \circ \varphi_{\tau *}$ is the push-forward of $J_0$ 
by $\varphi_\tau^{-1}.$  
\end{remark}

\begin{corollary}Under the conditions of Theorem \ref{ajax}, we have 
$$
\varphi_\tau^*\left({\mathcal O}_{(M,J_0)}\right)= {\mathcal O}_{(M,J_t)},
$$
where ${\mathcal O}_{(M,J)}$ denotes the structure sheaf of $M$ with respect to the complex structure $J$.
\end{corollary}

\begin{corollary} Let $g\in C^{\infty}(M)$ be a, not necessarily $J_0$-holomorphic function and let $\{z^i\}_{i=1,\dots,n}$ be local 
$J_0$-homolorphic coordinates on $M$. For $T$ as in Theorem \ref{ajax} and $\tau\in D_T$, we have
\begin{equation}\label{merkl}
\varphi_\tau^* g (z^1,\dots,z^n,\bar z^1,\dots,\bar z^n) = g(z_\tau^1,\dots,z_\tau^n,\bar z^1_\tau,\dots,\bar z^n_\tau),
\end{equation}
where $\{z_\tau^i\}_{i=1,\dots, n}$ are defined in Theorem \ref{lemmalocal1} and 
\begin{equation}\label{draghi}
\bar z_\tau = \overline{z_\tau} = e^{\bar \tau X}\cdot \bar z, 
\end{equation}
with $\bar z= (\bar z^1,\dots,\bar z^n), \bar z_\tau= (\bar z^1_\tau,\dots,\bar z^n_\tau)$.
\end{corollary}

\begin{remark}If $\tau=t\in \R$, then $\varphi_t$ in (\ref{mapa}) is the usual time $t$ flow of $X$ and is therefore $J_0$-independent.
For $\tau \notin \R$, $\varphi_\tau$ in general will depend on $J_0$. Moreover, if $\tau \notin \R$, $\varphi_\tau$ does not in general define 
a flow, in the sense that $\varphi_{\tau+\sigma}\neq \varphi_\tau \circ \varphi_\sigma$, even all if three of these diffeomorphisms are defined.
\end{remark}

We now note that if $|\tau|$ becomes too large it may happen that the complex structure $J_\tau$ no 
longer exists even though the functions $\{z^i_{\alpha,\tau}\}_{i=1,\dots,n}$ are well defined, and 
functionally independent, for all the open 
sets in the cover $\{U_\alpha\}$ of $M$. This happens, for instance, if for some $\alpha$ the $2n$-form 
$dz^1_{\alpha,\tau}\wedge \cdots\wedge dz^n_{\alpha,\tau}\wedge d\bar z^1_{\alpha,\tau}
\wedge \cdots\wedge d \bar z^n_{\alpha,\tau}$ has zeros. Note that

\begin{lemma}In the above notation, if $\tau\in\C$ is such that the functions $\{z^j_{\alpha,\tau}\}$ are well defined on an open set 
$U_\alpha$, then the $J_\tau$-holomorphic form $dz^1_{\alpha,\tau}\wedge \cdots\wedge dz^n_{\alpha,\tau}$ does not have zeros on $U_\alpha$. 
\end{lemma}

\begin{proof}
The action of the operator $e^{\tau X}$ {\it on holomorphic functions} can be inverted, since
$$
e^{-\tau X} z_{\alpha,\tau}^j = z_{\alpha}^j.
$$
\end{proof}

On the other hand, the zeros of $dz^1_{\alpha,\tau}\wedge \cdots\wedge dz^n_{\alpha,\tau}\wedge d\bar z^1_{\alpha,\tau}
\wedge \cdots\wedge d \bar z^n_{\alpha,\tau}$ mean that some linear combination of 
$dz_{\alpha,\tau}^j, j=1,\dots, n$, becomes real and the corresponding polarization\footnote{Recall that 
on a symplectic manifold $(M,\omega)$ a (geometric quantization) polarization, ${\mathcal P}$, is a 
Lagrangian distribution in the complexified tangent bundle $TM\otimes \C$. 
${\mathcal P}$ is said to be real if, at each point of $M$, ${\mathcal P}= \bar {\mathcal P}$. If 
$(M,\omega,J)$ is a K\"ahler manifold, then $T^{(1,0)}M$ defines a K\"ahler polarization ${\mathcal P}_J$, 
such that ${\mathcal P}_J \cap \bar{\mathcal P}_J=0$. (See, for example, \cite{Woodhouse}.)} becomes mixed.
For these values of $\tau$, 
which cannot be treated by a flow on $M_\C$ followed by projection to $M$, it is still interesting, 
for geometric quantization purposes, to study the associated mixed polarizations. In the next example, 
we illustrate this situation.

\begin{example}\label{plano}Consider the complex manifold $(\R^2,J_0)$, where $J_0$ is the homogeneous complex structure 
defined by the 
global holomorphic coordinate $z= x+\tau_0 y$, with ${\rm Im}\,\tau_0\neq 0$. Even though this is a noncompact example, it serves 
to illustrate the results above. Let $X=y\frac{\partial}{\partial x}$, so that 
$$
z_\tau = e^{\tau X} \cdot z = x + (\tau_0+\tau)y. 
$$
Letting $\tau_0=r_0+is_0, \tau=r+is, r_0,s_0,r,s\in \R$, we obtain that $\varphi_\tau:\R^2\to \R^2$ is the linear map defined 
by the matrix 
\begin{equation}\label{matrix}
\left[
\begin{array}{cc}
1 & r+r_0-\frac{(s+s_0)r_0}{s_0}\\
0 & \frac{s+s_0}{s_0} 
\end{array}\right]
\end{equation}
in the canonical basis.
This one-parameter family of diffeomorphisms defines, for $s\neq -s_0$, a linear isomorphism of $\R^2$. For 
$s=-s_0$, $\varphi_{r-is_0}$ maps 
$\R^2$ to a real dimension one linear subspace. Consider on the complex one-parameter subgroup of the formal 
group ${\rm Diff}(\R^2)_\C$,
$$
{\mathcal C}=\{e^{\tau X}, \tau\in \C\},
$$ 
the three subsets ${\mathcal C}= {\mathcal C}_+\cup{\mathcal C}_0\cup{\mathcal C}_-$ according to whether 
${\rm Im}(\tau_0+\tau)$ is 
positive, equal to zero or negative, respectively. Then, the map $e^{\tau X}\mapsto \varphi_\tau$ breaks up 
into maps
\begin{eqnarray}\nonumber
{\mathcal C}_+&\to& GL_+(\R^2)\subset Diff_{+}(\R^2)\\ \nonumber
{\mathcal C}_0&\to& M_2(\R) \subset C^\infty(\R^2,\R^2)\\ \nonumber
{\mathcal C}_{-}&\to& GL_{-}(\R^2)\subset Diff_-(\R^2).
\end{eqnarray}

Note that when $\tau$ is such that $s=-s_0$, we still have a well defined, albeit real, function
$$
z_{\tau}= x+(r_0+r)y = \bar z_\tau.
$$
At this value of $s$, the complex polarization associated to $J_\tau$ becomes real, see 
Example \ref{planosymp}.
\end{example}

\begin{remark}
Under the conditions of Theorem \ref{ajax}, note that for $J_0$-holomorphic $(k,0)$-forms
and $|\tau|$ sufficiently small we have
$$
\varphi_{\tau}^*\, f(z)\, dz^{i_1}\wedge\cdots\wedge dz^{i_k} =:  e^{\tau {\mathcal L}_{X_h}} \left( f(z)
dz^{i_1}\wedge\cdots\wedge dz^{i_k}\right) =f(e^{\tau{X_h}}\cdot z)\, dz^{i_1}_\tau\wedge\cdots\wedge dz^{i_k}_\tau, 
$$
while for $J_0$-anti-holomorphic forms of type $(0,k)$,
$$
\varphi_{\tau}^* f(\bar z) \, d \bar z^{i_1}\wedge\cdots\wedge d \bar z^{i_k} =: e^{\bar \tau {\mathcal L}_{X_h}}  \left(f(\bar z)
d\bar z^{i_1}\wedge\cdots\wedge d\bar z^{i_k}\right) = f(e^{\bar \tau{X_h}}\cdot \bar z) \, d\bar z^{i_1}_\tau\wedge\cdots\wedge d\bar z^{i_k}_\tau.
$$
Note that these expressions are convergent power series in $\tau$ and $\bar \tau$ respectively.
\end{remark}

\begin{remark}
A geometric interpretation of $\varphi_\tau$ can be given in terms of the complexification of $M$, as in \cite{BLU}. 
(See also Section \ref{s2}.)
\end{remark}

\section{Complexification of Hamiltonian flows}

Let $U\subset\R^{2n}$ be an open set. Let $J_0$ be the standard complex structure on $\R^{2n}\cong \C^n$ and
consider now $(U,J_0)$ with real analytic symplectic form $\omega$ such that $(U,\omega,J_0)$ is a K\"ahler manifold. 
For $h\in C^{\rm an}(U)$, let $X_h$ be its (real analytic) 
Hamiltonian vector field with flow $\varphi_t$. In this Section, we will consider Theorem \ref{ajax} in the context 
of complexified Hamiltonian flows and will study some additional properties valid in this case. 
Letting $z=(z^1,\dots,z^n)$ be local $J_0$-holomorphic coordinates on $U$, recall that $\omega$ is a type $(1,1)$-form.

\begin{theorem}\label{shakira}
Let $V\subset U$ be an open set and let $f\in C^{\rm an}(V)$ such that $e^{\tau X_h}$ can be applied to $f$ (see Definition \ref{canbe}) for $|\tau|<T$ for some $T>0$.
Then, the real-analytic Hamiltonian vector field of $e^{\tau X_h}\cdot f$ with respect to $\omega$ is given by the
Lie series 
\begin{equation}\label{xxx}
X_{e^{\tau {X_h}}\cdot f} = \sum_{k=0}^\infty \frac{\tau^k {\mathcal L}_{X_h}^k(X_f)}{k!} =:e^{\tau {\mathcal L}_{X_h}} X_f, \,\,\,|\tau|<T.
\end{equation}
\end{theorem}

\begin{proof}
Observe that, for $k\in \N$, $g\in C^\infty(V)$, $$\left({\mathcal L}_{X_h}^k X_f\right)(g)=X_{X_h^k(f)}(g).$$
From real analiticity of $\varphi_t$, for real time $t$, we then have 
\begin{equation}\label{xxxx}
X_{e^{t {X_h}}\cdot f} = e^{t {\mathcal L}_{X_h}} X_f .
\end{equation}
The equality (\ref{xxx}) follows by taking the analytic continuation of 
(\ref{xxxx}) in $t$.
\end{proof}

\begin{remark}\label{notsympl}
For real time $t$, as long as the flow of $X_h$ is defined, (\ref{xxxx}) is the statement, for the symplectomorphism $\varphi_t$ of 
Theorem \ref{ajax},
of the equality $(\varphi_t^{-1})_* X_f = X_{\varphi_t^* f}$. However, note that for ${\rm Im}\tau\neq 0$, 
$\varphi_\tau$  will not, in general, be a symplectomorphism. In this case, for $f$ not $J_0$-holomorphic, we have, in general,
$\varphi_\tau^* f\neq e^{\tau X_h}\cdot f.$ 
\end{remark}

\begin{remark}
Note that 
$\bar z_\tau = \overline{(e^{\tau X_h} \cdot z)}= e^{\bar \tau X_h}\cdot \bar z \neq e^{\tau X_h}\cdot \bar z=\bar z_{\bar \tau}$. 
\end{remark}

We have, for the Hamiltonian vector fields for the holomorphic coordinate functions,

\begin{lemma}  \label{liederivative}
Under the conditions of Theorem \ref{shakira},
$$
e^{\tau {\mathcal L}_{X_h}} X_{z^i} = X_{z_\tau^i},\,\,e^{\bar \tau {\mathcal L}_{X_h}} X_{\bar z^i} = X_{\bar z_\tau^i}
$$
where the left hand sides are convergent power series (in $\tau$ and $\bar \tau$ respectively).
\end{lemma}

\begin{proof}
We have ${\mathcal L}_{X_h}\omega =0$ and
$$
dz^i_\tau = d e^{\tau {X_h}} \cdot z^i = e^{\tau {\mathcal L}_{X_h}}  dz^i = 
e^{\tau {\mathcal L}_{X_h}} i_{X_{z^i}} \omega = i_{e^{\tau {\mathcal L}_{X_h}}X_{z^i} }\omega,
$$
so that $X_{z^i_\tau}=e^{\tau {\mathcal L}_{X_h}}X_{z^i}, i=1,\dots,n $. Since $z^i_\tau$ is given by an 
absolutely and uniformly convergent on compact subsets Lie series, all 
the expressions make sense as convergent series.
\end{proof}

\section{Action on K\"ahler structures}

\begin{theorem}\label{kahlerfamily}
Let $(M,\omega,J_0,\gamma_0)$ be a real analytic compact K\"ahler manifold with K\"ahler form $\omega$, 
complex structure $J_0$ and Riemannian metric $\gamma_0$. 
Assume $\omega,J_0,\gamma_0$ are analytic. Let $h\in C^{\rm an} (M)$. 

There exists $T>0$ such that for each 
$$
\tau\in D_T :=\{\tau\in\C: |\tau| <T\}
$$
one has that:
\begin{itemize}
\item[(i)]
$(M,\omega,J_\tau)$ is a K\"ahler manifold, where $J_\tau=\varphi_{\tau}^*J_0$ and 
$\varphi_\tau$ is the biholomorphism of Theorem \ref{ajax}. 

\item[(ii)] Let 
$\kappa_0$ be a local analytic K\"ahler potential for $(M,\omega,J_0)$. A local K\"ahler potential for 
$(M,\omega,J_\tau)$ is then given by 
\be
\label{formula0}
\kappa_\tau=-2{\rm Im}\,\, \psi_\tau, 
\ee
with
\begin{equation}\label{formula}
\psi_\tau = -\frac{i}{2}\,e^{\tau X_h}\cdot \kappa_0 + \tau h - \alpha_\tau ,
\end{equation}
where $\alpha_\tau$ is the analytic continuation in $t$ of 
\be
\label{alfa}
\alpha_t = \int_0^t \, e^{t' X_h}(\theta(X_h)) \, dt' \ ,
\ee
$\theta$ is the real local potential for $-\omega$ defined by 
$\theta^{(0,1)}=\frac{i}{2}\bar\partial_0 \kappa_0$, with $\bar\partial_0$ the $\bar\partial$-operator 
relative to the complex struture $J_0$.  
\end{itemize}
\end{theorem}

\begin{remark} One has $d\alpha_\tau= e^{\tau d\circ {i}_{X_h}}\theta -\theta$ and also
$$
\alpha_\tau = \sum_{k=1}^\infty \frac{\tau^k}{k!} X_h^{k-1}(\theta(X_h)),
$$
for $\alpha_\tau$ in (\ref{alfa}). 
\end{remark}

\begin{proof}
To prove $(i)$ it is enough to show that $\omega$ is of type $(1,1)$ with respect to $J_\tau$. Positivity, 
for small enough $|\tau|$, follows from continuity at $\tau=0$. From Lemma \ref{liederivative},
$$
X_{z_{\tau}^i} = e^{\tau {\mathcal L}_{X_h}} X_{z^i}.
$$
Note that the local vector fields $\{X_{z_{\tau}^i}\}_{i=1,\dots,n}$ are linearly independent at every point of 
an open neighborhood. 
Since ${\mathcal L}_{X_h} \omega=0$ and $\omega$ is of type $(1,1)$ with respect to $J_0$, we obtain
$$
\{z_{\tau}^i,z_{\tau}^j\}=0, i,j=1,\dots,n.
$$
Therefore,
$$
dz_{\tau}^i (X_{z_{\tau}^j}) = \omega (X_{z_{\tau}^i},X_{z_{\tau}^j})=0,
$$
so that $\{X_{z_{\tau}^i}\}_{i=1,\dots,n}$ generate the $(0,1)$ tangent space with respect to $J_\tau.$
Therefore, $\omega$ is of type $(1,1)$ with respect to $J_\tau$ which concludes the proof of $(i)$.

To prove $(ii)$, let $\theta = \theta^{(1,0)_\tau}+\theta^{(0,1)_\tau}$, 
be the decomposition of $\theta$ into $(1,0)$ and $(0,1)$ parts with 
respect to $J_\tau$. Since $\omega$ is of type $(1,1)$ with respect to $J_\tau$, 
$\bar\partial_\tau \theta^{(0,1)_\tau}=0$, where $\partial_\tau, \bar\partial_\tau$ are the $\partial,\bar\partial$-operators 
with respect to $J_\tau.$ Therefore, by the $\bar \partial$-lemma, $\theta^{(0,1)_\tau}=\bar\partial_\tau \psi_\tau$ for 
some locally defined complex valued analytic function $\psi_\tau.$ 
A K\"ahler potential for $\omega$ with respect to $J_\tau$ will then be given by $\kappa_\tau = -2{\rm Im}\,\psi_\tau$. 
Recall that $\{X_{z_{\tau}^i}\}_{i=1,\dots,n}$ is a basis of the $(0,1)$ tangent space for $J_\tau$. 
Since $\omega=-d\theta$, we have, expanding in powers of $\tau$,
$$
e^{\tau {\mathcal L}_{X_h}} \theta = \theta -\tau dh + d\alpha_\tau.
$$
Then, since $\theta^{(0,1)} =-\frac{i}{2}\bar\partial_0 \kappa_0$ and $X_{z_{\tau}^i}=e^{\tau {\mathcal L}_{X_h}}X_{z^i}$, 
the validity of $(ii)$ follows from the equality
$$
\theta (X_{z_{\tau}^i}) =-\frac{i}{2} X_{z_{\tau}^i}(e^{\tau X_h}\cdot \kappa_0) + \tau 
dh(X_{z_{\tau}^i}) - d\alpha_\tau (X_{z_{\tau}^i}), i=1,\dots,n,
$$
where we use $e^{\tau{\mathcal L}_{X_h}}\theta (X_{z_{\tau}^i})=e^{\tau{\mathcal L}_{X_h}}
\theta(e^{\tau{\mathcal L}_{X_h}}X_{z^i})=e^{\tau X_h}\theta(X_{z^i})$ due to the fact that, as a consequence of Remark 
\ref{autom}, $e^{\tau {\mathcal L}_{X_h}}$ acts as a (local) automorphism of the algebras of tensor fields on $M$.
\end{proof}

Note that $\varphi_t$ defines a symplectomorphism  of $(M,\omega)$ for all $t\in \R$. However, 
for complex $\tau$, $\varphi_\tau$ will not, in general, be a symplectomorphism, as already 
noted in Remark \ref{notsympl}. 
Therefore, the 
K\"ahler structures  $(M,J_0,\omega,\gamma_0)$ and $(M,J_\tau,\omega,\gamma_\tau)$ will, in general, be non-equivalent. 

That the evolution in real time does not change the diffeomorphism equivalence class of the K\"ahler structure, 
can be checked explicitly by verifying that under the flow of the Hamiltonian symplectomorphisms $\varphi_t, t\in \R$
the K\"ahler potential just evolves by composition with the flow. In fact,

\begin{proposition}\label{realnada}
Under the conditions of Theorem \ref{kahlerfamily}, let $\tau_0\in D_T, s\in \R$ such that $\tau+s \in D_T$. We then have,
$$
\kappa_{\tau+s}=\varphi_s^*\kappa_\tau.
$$
\end{proposition}

\begin{proof}
From formula (\ref{alfa}) it is easy to check that, for real $t,s$, one has
$$
\varphi_s^* \alpha_{t} = e^{sX_h}\cdot \alpha_t = \alpha_{t+s}-\alpha_s.
$$
By analytic continuation in $t$, we obtain, 
$$
e^{sX_h}\alpha_\tau = \alpha_{\tau +s}-\alpha_s.
$$ 
Since, for real $s$, we have $\alpha_s = \bar \alpha_s$, and using $X_h(h)=0$, it is straightforward to verify that
$$
\kappa_{\tau +s} = e^{sX_h} \kappa_\tau.
$$
\end{proof}

If $M$ is not compact or if $|\tau|$ becomes too large, in general the maps $e^{\tau X_h}$ may 
develop singularities, even if they stay as local biholomorphisms on large subsets of $M$. 
Also, on some subsets of $X$ one may lose the positivity of the metric $\gamma_\tau$. Nevertheless, 
it can still be very useful to consider these maps and their action on the set of 
polarizations of the symplectic manifold $(M,\omega)$. Note that the action of the diffeomorphisms 
$\varphi_\tau$ on polarizations is given by the following 

\begin{theorem} Under the conditions of Theorem \ref{kahlerfamily}, 
the K\"ahler polarization (that is the $(0,1)$-tangent space) 
for $(M,\omega,J_\tau)$ is given by
\be
\label{formula2}
{\mathcal P}_\tau = (\varphi_{\tau}^{-1})_* {\mathcal P}_0 = e^{\tau {\mathcal L}_{X_h}} {\mathcal P}_0,
\ee
where this expression can be interpreted as a convergent power series in $\tau$ 
if $e^{\tau {\mathcal L}_{X_h}}$
is applied to appropriate sections of ${\mathcal P}_0$ such as 
$X_{z^i}, i=1,\dots,n$, where $\{z^i\}_{i=1,\dots,n}$ are local holomorphic coordinates on $(M,J_0)$.
\end{theorem}

\begin{proof}
The $(0,1)$-tangent space with respect to $J_\tau$ has a basis 
$\{X_{z_{\tau}^i}\}_{i=1,\dots,n}$, where $X_{z_{\tau}^i}=e^{\tau {\mathcal L}_{X_h}} 
X_{z^i}$ and $\{X_{z^i}\}_{i=1,\dots,n}$ is a basis of the $(0,1)$-tangent space 
with respect to $J_0.$
\end{proof}

With a view to the application to half-form quantization, note also that we have the following result.

\begin{proposition}
Under the conditions of Theorem \ref{kahlerfamily}, the fiber of canonical bundle $K_\tau$ of $(M,J_\tau)$ 
is generated at each point by
$$
\Omega_\tau = e^{\tau {\mathcal L}_{X_h}} \Omega_0,
$$
where $\Omega_0$ generates the fiber of the canonical bundle $K_0$ of $(M,J_0)$ and where 
the expression makes sense as a convergent Lie series if $e^{\tau {\mathcal L}_{X_h}}$ is 
applied to appropriate local sections of $K_0$ such as $\Omega_0= dz^1\wedge \cdots \wedge dz^n$.
\end{proposition}

\begin{proof}
We only need to note that
$$
dz_{\tau}^i= d e^{\tau {X_h}}\cdot z^i =  e^{\tau {\mathcal L}_{X_h}} dz^i,
$$
where the equality is valid for real $\tau$ and follows for complex $\tau$, with small enough $|\tau|$, by 
analytic continuation. 
\end{proof}

\section{Geometric quantization interpretation}
\label{gqsection}

For $\frac{[\omega]}{2\pi}\in H^2(M,\Z)$ integral, let $L\to M$ be a prequantum line 
bundle with connection $\nabla$ such that 
$F_\nabla=-i \omega$. Let $\hat f$ denote the prequantization
of $f \in C^{\rm an}(M)$ acting on sections of $L$,
$$
\hat f = i \nabla_{X_f} + f =  i X_f + f-\theta(X_f),
$$
where, as in $(ii)$ of Theorem \ref{kahlerfamily}, $\theta$ is a local real potential for $-\omega$.
In this case, there is a clearer interpretation of 
the action on the K\"ahler potentials as $e^{i\psi_\tau}$ can be considered as a section of $L$ 
(written in the unitary frame corresponding to $\theta$) such that 
$e^{\kappa_\tau}$ is a section of $L\otimes L^{-1}.$ Then,
\begin{theorem}
Let $\psi_\tau$ be as in (\ref{formula}) in Theorem \ref{kahlerfamily}. Then, $e^{i \psi_\tau}$ is obtained by  
the analytic continuation to complex time of the prequantization of $e^{t X_h}$ acting on $e^{i\psi_0}$, that is 
\be
\label{formula3}
e^{i \psi_\tau} = e^{-i \tau \hat h} \, (e^{i \psi_0}) = e^{i \, (e^{\tau X_h}\psi_0 + \tau h - \alpha_\tau)}.
\ee
\end{theorem}

\begin{proof}The formula is valid for real time $\tau=t\in \R$. To see this, we only need to show that
$$
\frac{d}{dt}e^{i\psi_t} = i\dot\psi_t e^{i\psi_t} = -i \hat h e^{i\psi_t}.
$$
This follows by direct differentiation and by using $\dot\alpha_t = e^{tX_h}\theta(X_h)$ and $X_h \cdot \alpha_t= e^{tX_h}
\theta(X_h)-\theta(X_h)$ where, again, the expression is valid in a local unitary frame of $L$  associated to $\theta$. 
The theorem then follows by analytic continuation in $t$.
\end{proof}

Note that this is the analytic continuation to complex time of the usual formula for the evolution of  
wave functions in geometric quantization, under the Hamiltonian flow of $X_h$. 
(See, for example, Chapters 8 and 9 in \cite{Woodhouse}. For other examples of the validity (in complex time) of the 
same formula, see \cite{KMN1}.) For a related description see also \cite{Do3}.

\begin{remark}
The reason to have $e^{i \psi_0}$ evolve with $e^{-i \tau \hat h}$ rather than with (a lifting to $L$ of) $\varphi_\tau$
is due to the fact that a ${\mathcal P}_0$-polarized section of $L$ evolves under the flow to a 
${\mathcal P}_\tau$-polarized section. Then,
the product of 
a $J_0$--holomorphic section $F_1$ with a $J_0$-antiholomorphic $F_2$
evolves with
$$
F_1^\tau \, F_2^\tau  =  e^{-i\tau \hat h} (F_1) \, e^{i \bar \tau \hat h}(F_2),
$$
thus corresponding the evolution of $e^{-\kappa_0}$ above.
\end{remark}

\section{Symplectic and complex pictures}

Let $(M,\omega,J_0)$ be a compact K\"ahler  manifold.
In \cite{Do1} (see also \cite{Do2,Do3}), Donaldson has considered a formal ``complexification'' of the 
infinite-dimensional group of Hamiltonian
symplectomorphisms of  $(M,\omega)$ and related it with geodesics in the space of  
K\"ahler metrics with fixed cohomology class
and
Mabuchi metric \cite{M}. In the present section, we will recall some aspects of this description and will
explain in which sense it can be extended in an interesting way for geometric quantization. As the section is
mainly illustrative, and intended to give just a general qualitative perspective, 
we will not consider any technicalities 
arising from infinite-dimensions. 

\subsection{Complex picture}

 Let $G$ denote the 
group of Hamiltonian symplectomorphisms of $(M, \omega)$ and $G_\C$ its formal complexification. Following Donaldson \cite{Do1}, in the 
present paper we  are studying two types of ``orbits'' of  $G_\C$. The orbit of first type
for the pair $(\omega, J_0)$ coincides with 
the space of K\"ahler metrics on $(M, J_0)$ with fixed cohomology class
\bas
\hoj &=& \left\{\varphi^* \omega_0, \, \, \varphi \in {\rm Diff}(M), [\varphi^*\omega] = [\omega] , \,  
\hbox{the pair $(\varphi^*\omega, J_0)$ is K\"ahler} \right\}  \cong \\
&\cong & \left\{\eta \in C^{\infty}(M) \ : \ \omega + i \partial_0 \overline{\partial_0} \eta >0 \right\} / \R .
\eas
where $\partial_0$ is the $\partial$-operator relative to the complex structure $J_0$.
Morally
\be
\label{orh}
\hoj = (\gco^* \cdot \omega, J_0)
\ee
so that, for $(\omega, J_0)$, $\gco$ corresponds to the following subset (not subgroup)
of ${\rm Diff}(M)$,
$$
\gco^{(\omega, J_0)} = \left\{\varphi  \in {\rm Diff}(M) \ : \ \varphi^*\omega \in \hoj\right\}
$$
and
$$
\hoj \cong \gco^{(\omega, J_0)} / G
$$
has the structure of an infinite-dimensional symmetric space with constant
negative curvature \cite{Se,Do1} for the Mabuchi metric \cite{M}. 
The transitivity of $\gco$ in (\ref{orh}) is a consequence of Moser's theorem. Choosing a path, $t \mapsto \omega_t, \, t\in [0,1]$, between two symplectic structures in $\hoj$, there are diffeomorphisms $\varPhi_t : M\to M$, obtained by integrating a 
$t$-dependent vector 
field on $M$,  such that $\varPhi_t^*\omega_t=\omega_0$. 
Notice that the family of Moser maps is unique up to composition with a family of symplectomorphisms of $(M,\omega)$.
We will
call the study of  orbits $\hoj$ the ``complex picture'' as 
the complex structure is kept constant. 

The Mabuchi metric on the space of K\"ahler potentials reads
$$
||\delta\phi||^2 = \int_M |\delta \phi|^2 \frac{\omega^n}{n!},
$$
where $\dim M = 2n$ and $\delta\phi\in C^\infty(M)$ is a tangent vector at the metric 
$\gamma_\phi=\omega_\phi(\cdot, J_0 \cdot).$ This 
metric induces a metric on $\hoj$, 
$$\hoj \cong \left\{\eta \in C^{\infty}(M) \ : \ \omega + i \partial_0 \overline{\partial_0} \eta >0, \, \, \int_M \eta \, \omega^n = 0 \right\}.$$

Consider a path 
of K\"ahler potentials $\phi_t$, with $\phi_0=0$, corresponding to the K\"ahler forms
$$
\omega_t = \omega_0 + i \partial_0 \bar\partial_0 \phi_t.
$$
The condition for this path to be a geodesic in the above metric reads
$$
\ddot \phi = \frac12 ||\nabla \dot\phi||^2_\phi,
$$
where $||\nabla \dot \phi||^2_\phi$  is the squared norm of the gradient of $\dot\phi$ with respect to the K\"ahler 
metric determined 
by $\phi$.

\subsection{Symplectic picture}

The second type of ``$\gco$ orbits''
corresponds to acting  on the complex structure
with fixed symplectic structure
\ba
\goj &=& (\omega, G_{\C*} \cdot  J_0)    =  \nonumber \\
&=& \left\{(\omega, \varphi_* J_0), \, \varphi \in \gco^{(\omega, J_0)} \right\} \label{org}
\ea
If $Aut(M, J_0)$ is discrete, than the ``action'' of $\gco$ is free and 
$$
\goj \cong \gco .
$$
In that case we get a natural map
\ba   
\nonumber \pi  :  \goj & \longrightarrow & \hoj \\
(\omega, \varphi_*J_0)  & \mapsto & \varphi^*(\omega, \varphi_*J_0) = (\varphi^* \omega , J_0)   \ ,
\label{proje}
\ea
giving $\goj$ the structure of a principal $G$--bundle over $\hoj$. 
 
If $Aut(M,J_0)$ has continuous subgroups, then 
$$
\goj \cong G_\C / Aut_0(M,J_0),
$$ 
where $Aut_0(M,J_0)$ is the connected component at the identity of  $Aut(M,J_0)$. Instead of (\ref{proje})
one then gets the diagram
\be
\label{diagram10} 
\mbox{
\xymatrix{
&&G_\C    \ar[ld]_{\pi_1} \ar[rd]^{\pi_2}  &  \\
& \goj  && \hoj}}  
\ee
We see that 
a map from 
orbits of imaginary time one-parameter Hamiltonian subgroups
in $\goj$ to $\hoj$ is defined if and only if these orbits do not
correspond to subgroups of $Aut(M,J_0)$.

The study of  $\goj$ is natural in geometric quantization, where it is 
natural to fix a symplectic structure and study the dependence of 
quantization on the choice of complex structure (even for complex structures related
by a symplectomorphism). We call the study of $\goj$ the ``symplectic picture''.

\subsection{Connecting the pictures and analytic Cauchy problem}

Let us, for simplicity, suppose in this section that $Aut(M, J_0)$ is discrete.
To go from the  symplectic picture to the complex picture we have the
projection $\pi$ in (\ref{proje}). To go from the complex picture 
to the symplectic picture we need a section of  (\ref{proje}). 
If $G$ was a compact Lie group,
a natural global section would be given by lifting the action of 
$\exp(i {\rm Lie}(G))$ on $\omega$ to $J_0$. It turns out, 
however, that the Cauchy problem for the geodesics in $\hoj$ does not always 
have solution, which means that there are (non-analytic)
$h \in C^\infty(M)$ for which $e^{it {\cal L}_{X_h}} \omega$,
is not well defined for any $t \neq 0$ \cite{Do1}. On the other hand,
if we restrict the $G$-bundle in (\ref{proje}) to the real-analytic case, 
 \ba   
\nonumber \pi^{\rm an}  :  \goj^{\rm an} & \longrightarrow & \hoj^{\rm an} \\
(\omega, \varphi_*J_0)  & \mapsto & \varphi^*(\omega, \varphi_*J_0) = (\varphi^* \omega , J_0)   \ ,
\label{projean}
\ea
where $\varphi\in {\rm Diff}(M)^{\rm an}\cap G_\C$, then, for each vector in the tangent space at
$(\omega, J_0)$ in $\hoj^{\rm an}$, the Cauchy problem has a solution given by acting on 
$\omega$ with imaginary time analytic
Hamiltonian symplectomorphisms. This is a consequence of Proposition \ref{geod}. More explicitly,
there are sections $\sigma$ of (\ref{projean}) 
defined on neighborhoods of $0$ of every direction in $T_{(\omega,J_0)}\hoj^{\rm an}$ given by
$$
\sigma( \varphi_{is}^* \omega, J_0) = (\omega, (\varphi_{is})_* \ J_0) , \forall h\in C^{\rm an}(M),  
$$
 where $|s|<T_h$ and  $T_h\in \R$ is an $h$-dependent positive real number. 
Notice, however, that the compactness of $M$
is crucial for this result to hold. (See Example \ref{ss32} where we consider a complete real analytic 
Hamiltonian vector field 
on the plane for which the Cauchy problem has no solution.) 

Consider now a path $(\omega_t,J_0)$ in $\hoj$ starting at $(\omega_0,J_0)=(\omega_,J_0)$. 
Fixing a family of Moser maps $\varPhi_t\in {\rm Diff}(M)$, with $\varPhi^*_t\omega_t=\omega$, 
corresponds to fixing a lift $(\omega,J_t)$ of the path to $\goj$, 
with $J_t= \varPhi_t^*J_0$. 
Of course, in this way we obtain equivalent K\"ahler structures 
$$
(M,\omega,J_t,\gamma_t)=\varPhi_t^*(M,\omega_t,J_0,\tilde \gamma_t).
$$

If we denote by $\partial_t$ the $\partial$-operator relative to the complex structure $J_t$ we then have, 
for a (local) K\"ahler potential $k_t$ for $\omega$ with respect to $J_t$,
\begin{equation}\label{moser}
\omega = i\partial_t\bar\partial_t k_t = \varPhi_t^* \omega_t = \varphi_t^* (\omega+i\partial_0\bar\partial_0 \phi_t).
\end{equation}
It follows that 
\begin{equation}\label{phitokappa}
\phi_t = k_t \circ \varPhi_t^{-1} -k_0,
\end{equation}
where $k_0$ is a (local) K\"ahler potential for $\omega$ with respect to $J_0$.

\section{Examples of extension beyond the space of K\"ahler structures}
\label{ss23}

\begin{example}\label{planosymp}Consider again Example \ref{plano} where now $\R^2$ is equipped with its standard symplectic 
structure $\omega=dx\wedge dy$, so that $(\R^2,J_0,\omega)$ is a K\"ahler manifold. 
The diffeomorphisms $\varphi_\tau$ (for $s \neq -s_0$) in this case were generated by $X=y\frac{\partial}{\partial x}=X_h$, 
where $h(x,y)=\frac{y^2}{2}$.

We have the family of K\"ahler polarizations of $(\R^2,\omega)$,
$$
{\mathcal P}_{\tau} = \langle \frac{\partial}{ \partial \bar z_\tau}\rangle, s > -s_0,
$$
where we recall that $\tau_0=r_0+is_0, \tau=r+is, r_0,s_0,r,s\in \R$.
When $s\to -s_0$, we see that ${\mathcal P}_{\tau} \to {\mathcal P}_{{r-is_0}}$ where 
$$
{\mathcal P}_{{r-is_0}}=\left\{-(r_0+r)\frac{\partial}{\partial x}+\frac{\partial}{\partial y}\right\}_\C
$$
is a real polarization.

It is straightforward to obtain the metric $\gamma_\tau$ defined by $(\omega,J_\tau)$,
$$
\gamma_\tau = \left(\frac{1}{s_0+s}\right)dx^2 + 2\left(\frac{r_0+r}{s_0+s}\right) dxdy + 
\left(\frac{(s_0+s)^2+(r_0+r)^2}{s_0+s}\right) dy^2,
$$ 
where $\tau_0=r_0+is_0, \tau=r+is$. Note that as $s\to -s_0$ the map $\varphi_\tau$ in (\ref{matrix}) becomes singular. In this 
limit, after a rescaling by $(s_0+s)$, the metric becomes 
$$
\lim_{s\to -s_0} (s_0+s)\gamma_{\tau} = d(x + (r_0+r) \,y)^2,
$$
so that there is a metric collapse of $\R^2$ to $\R$, with a degeneration along the kernel of the linear map 
$\varphi_{-is_0}$.
\end{example}

\begin{example}
\label{ss32}
Consider, again, the symplectic plane with standard symplectic structure and with initial K\"ahler structure given 
by the standard holomorphic coordinate $z=x+iy$. Consider the Hamiltonian function 
$$
h(x,p)= \frac12 (xy)^2,
$$ 
with Hamiltonian vector field 
$$
X_h =  xy\left(x \frac{\partial}{\partial x} -y \frac{\partial}{\partial y}\right). 
$$

Let $\varphi_{it}$ be the map defined in Theorem \ref{ajax}, where $t\in \R$.

\begin{lemma}\label{x2p2}
$$
z_{it}:= e^{it X_h} z_0 = e^{itxy} x + ie^{-itxy}y.
$$
\end{lemma}

\begin{proof} This follows easily from direct computation from 
$X_h^{2k}(z_0)= (xy)^{2k} z_0$ and $X_h^{2k+1}(z_0)= (xy)^{2k+1} \bar z_0$, for $k=0,1,2,\dots.$
\end{proof} 

\begin{remark}As we will see below, in this (non-compact) case, the map $\varphi_{it}$ for $t\neq 0$ never defines a 
global diffeomorphism of the plane. However, in the neighborhood of each point, for $|t|$ small enough, 
$\varphi_{it}$ will be a local diffeomorphism. 
\end{remark}

\begin{lemma}The family of K\"ahler potentials and K\"ahler forms are given by
\begin{equation}\label{kkkk}
\kappa_{it} = \frac12 \cos(2txy) (x^2+y^2)+ tx^2y^2
\end{equation}
and
\begin{equation}\label{kkkkk}
\omega = i g_{it}(x,y) dz_{it} \wedge d\bar z_{it},
\end{equation}
where 
$$
g_{it}^{-1} = 2t\left( x^2+ y^2-2xy \sin (2txy) \right)+ 2\cos (2txy).
$$
\end{lemma}

\begin{proof} 
The formula for $k_{it}$ follows from (\ref{formula0}) and (\ref{formula}) by noticing that, with $\theta = -xdy$, one has 
$X_h(\theta(X_h))=0$ so that $\alpha_{it}=2ith$.
The formula for $g_{it}$ follows by computing of $dz_\tau$ and $d\bar z_\tau$ from Lemma \ref{x2p2} and comparing 
with $\omega = dx\wedge dy$.
\end{proof}

We see that for any value of $t$ there is a sufficiently small neighborhood of the origin $x=y=0$ 
where $g_{it}$ is positive and the metric is K\"ahler. On the other hand, for any $t\neq 0$ there are regions where 
$g_{it}$ becomes negative. Let us consider the case $t>0$. (The $t<0$ case is clear.) We then have
$$
g_{it}^{-1}\geq 2t {\rm min} \{(x-y)^2,(x+y)^2\} + 2\cos(2txy).
$$
We see that negative values of $g_{it}$ can arise only in the region
$$
\{(x,y)\in\R^2: |x-y|<\frac{1}{\sqrt{t}}\} \cup \{(x,y)\in\R^2: |x+y|<\frac{1}{\sqrt{t}}\}.
$$
As $t$ increases, these two strips become more and more concentrated around $x=y$ and $x=-y$.
To find regions where $g_{it}$ becomes negative, let us consider $x=y$ and $2txy=\frac{\pi}{2}+2k\pi, k\in \N$.
At those points we have $g_{it}^{-1}=0$, implying that $\varphi_{it}$ is not a global diffeomorphism. 
By Taylor expanding around one of those points, we see that for 
$2txy$ bigger but sufficiently close to $\frac{\pi}{2}+2k_0\pi, k_0\in \N$, the sign of $g_{it}$ is negative. 
Therefore, for any value of $t> 0$, there is a countable set of open regions 
where the polarization ${\mathcal P}_{it}$ is pseudo-K\"ahler. 
This shows that the Cauchy problem for the geodesic
starting at $(\omega,J_0)$ and in the direction of $h$ does not have a solution.  

\begin{remark}
Note that the polarizations ${\mathcal P}_{it}$ on the boundary of the K\"ahler regions become real, since 
$\left(\frac{\partial (z_{it},\bar z_{it})}{\partial (x,y)}\right)$, whose determinant equals $g_{it}^{-1}$, 
has rank one at those points.
\end{remark}
\end{example}

\section{Relation with the Burns-Lupercio-Uribe approach}
\label{s2}

In \cite{BLU}, the authors also study the complexification $\hmc(M,\omega)=G_\C$ 
when $(M,\omega)$ is real analytic.
In this work, the Hamiltonian flow in complex time is described via an actual flow on a complexification 
$(M_\C,\omega_\C)$ of $M$. (For the complexification of real analytic Riemannian manifolds see 
\cite{GS1,GS2,LS,Szoke-91}.) While the formalism of \cite{BLU} is more geometric than the one we consider, 
it requires the introduction of $M_\C$ which may not be convenient when considering some real 
non analytic examples, as in subsection \ref{ssnonanal}.

\subsection{The case of diffeomorphisms}
\label{sscasediffeo}

Let us briefly review the \cite{BLU} construction and show that in the real analytic case and when 
$\varphi_\tau$ is a diffeomorphism it is equivalent to ours. 

Let $(M,\omega,J_0)$ be a K\"ahler manifold with real analytic $(\omega,J_0)$. 
The complexification $M_\C$ is a holomorphic symplectic manifold, with holomorphic symplectic form 
$\omega_\C$ given by the analytic continuation of $\omega$, as described below. On $M_\C$ there exists an anti-holomorphic involution $\sigma:M_\C\to M_\C$, such that 
$M\subset M_\C$ is the fixed point set of $\sigma$ and $\sigma^*\omega_\C=\bar\omega_\C$. 
Moreover, there is a holomorphic projection $\pi_0:M_\C\to M$, such that if $\iota: M\hookrightarrow M_\C$ 
is the inclusion, $\pi_0 \circ \iota = id_M.$ Let us describe this construction locally.

Let $\{x^j,y^j\}_{j=1,\dots,n=\dim M}$ be a local system of real analytic coordinates on a sufficiently 
small open set $U\subset M$. Their analytic 
continuation, $\{x^j_\C,y^j_\C\}_{j=1,\dots,n}$, then form a local system of holomorphic coordinates on 
the complexification $U_\C \subset M_\C$ of $U$. 
Let us assume that the real analytic coordinates are such that $\{(x^j,y^j)=({\rm Re}\,z^j, {\rm Im}\,z^j)\}_{j=1,\dots,n}$ 
where $\{z^j\}_{j=1,\dots,n}$ is a 
local system of 
$J_0$-holomorphic coordinates on $U$, 
so that $\omega$ is given by the type-$(1,1)$ form
$$
\omega = \sum_{j,k=1}^n i\,g_{j\bar k}(z,\bar z) dz^j \wedge d\bar z^k,
$$
where the functions $g_{j\bar k}= \overline{g_{k\bar j}}, j,k=1,\dots,n$, are real 
analytic and $z=(z^1,\dots,z^n)$. Then, one has holomorphic coordinates $(z_\C,w_\C)$ on $U_\C$ which are the 
analytic continuation of $(z,\bar z)$ respectively. Concretely, one has $z_\C^j=x_\C^j+iy_\C^j, 
w_\C^j=x_\C^j-iy_\C^j, j=1,\dots,n$. The holomorphic symplectic form on $U_\C$ reads
$$
\omega_\C = \sum_{j,k=1}^ni\, g^\C_{j k} (z_\C,w_\C) dz^j_\C \wedge dw^k_\C,
$$
where the holomorphic funtion
$g^\C_{j k}$ is the analytic continuation of $g_{j \bar k}$.

The anti-holomorphic involution 
$\sigma:U_\C\to U_\C$ is given by $\sigma (z_\C,w_\C)= (\bar w_\C, \bar z_\C)$ and the holomorphic projection 
$\pi_0:U_\C\to U$ is given by $\pi_0(z_\C,w_\C)= z_\C$. Note that both $\sigma$ and $\pi_0$ are 
well defined globally, since if
one has a holomorphic change of local coordinates on $M$, $\tilde z=F(z)$, $F$ holomorphic, then
also one has $\tilde z_\C = F (z_\C)$ and $\tilde w_\C = \bar F (w_\C)$.
The subset $U\subset U_\C$ is described by the equations 
$$
w_\C=\bar z_\C.
$$

Let now $h\in C^{\rm an}(M)$ be a real Hamiltonian function, let $\tau\in \C$ and consider $\tilde h= \tau h$.
\footnote{While in the present work we consider the complex time Hamiltonian flow of 
real Hamiltonian functions $h$, in \cite{BLU} the authors prefer to take more general 
complex valued Hamiltonians $\tilde h$
and to consider their flows in real time. Our formalism can be easily adapted to the case of general complex 
valued real analytic $h$. In particular, as it will be seen below in Section \ref{sectiong}, 
to describe geodesics on $\hoj^{\rm an}$ it suffices to consider 
the imaginary time flow real Hamiltonians or, equivalently, real time flow of (pure) imaginary Hamiltonians.}
Let $\tilde h_\C$ be its analytic continuation to $U_\C$ 
(if necessary one takes a smaller neighborhood), so that 
$$
\tilde h_\C(z_\C,w_\C)_{|_{\bar z_\C=w_\C}} = \tau h(z,\bar z), \, z=z_\C. 
$$
The Hamiltonian vector field of ${\rm Re}\,\tilde h_\C$ with respect to the real symplectic form 
$\frac12 (\omega_\C+\bar\omega_\C)$, is given by 
$$
X_{{\rm Re}\,\tilde h_\C}^{\frac12(\omega_\C+\bar\omega_\C)} = X^{\omega_\C}_{\tilde h_\C} +X^{\bar\omega_\C}_{{\overline{\tilde h}_\C}} 
= \tau X^{\omega_\C}_{h_\C} + \bar \tau X^{\bar \omega_\C}_{\bar h_\C}.
$$  
Let $I\subset \R$ be a sufficiently small open neighborhood of zero, and let $\eta_t, t\in I$ be the time-$t$ 
flow of the Hamiltonian vector field $-X_{{\rm Re}\,\tilde h_\C}^{\frac12(\omega_\C+\bar\omega_\C)}$.
Let ${\mathcal F}_0$ be the foliation of $U_\C$ defined by the fibers of the projection $\pi_0$. Denote by 
${\mathcal F}_t$ its push-forward by $\eta_t$, and let $\pi_t$, for $|t|$ small enough, be the corresponding 
projection to $U$. In \cite{BLU}, the authors consider the map (see their equation $(1.1)$)
\begin{equation}\label{blumap}
\pi_t\circ \eta _t \circ \iota : U\to M.
\end{equation}

If $M$ is compact as we are assuming, as shown in \cite{BLU}, this globalizes to a well defined 
diffeomorphism of $M$, for small enough $|t|$. 

We then have,
\begin{proposition}\label{blue} Under the conditions of Theorems \ref{ajax} and \ref{kahlerfamily}, 
the diffeomorphisms of $M$ defined in 
(\ref{mapa}) and (\ref{blumap}) are related by
$$
\varphi_{t\tau}^{-1} = \pi_t\circ \eta_t,
$$
or, equivalently, the following diagram is commutative

\be
\label{diagram} 
 \mbox{
\xymatrix{
&&M_\C    \ar[rr]^{\eta_t}&& M_\C \ar[d]^{\pi_t} \\
&&(M,J_0) \ar @{^{(}->}[u]^{\iota}   && (M,J_{t\tau}) \ar[ll]^{\varphi_{t\tau}} .}}  
\ee

\end{proposition}

\begin{proof}
Consider local coordinates $(z_\C,w_\C)$ on $M_\C$ as above, and assume we are under the conditions 
and notations of Theorems \ref{ajax} and \ref{kahlerfamily}. Let $(z,\bar z)\in M$. Then,
$$
{{\mathcal F}_t}_{|_{\eta_t(z,\bar z)}} = \eta_t\left({{\mathcal F}_0}_{|_{(z,\bar z)}} \right).
$$
The leaf ${{\mathcal F}_0}_{|_{(z,\bar z)}}$ is determined by the equations $z_\C=z$, so that the leaf 
${{\mathcal F}_t}_{|_{\eta_t(z,\bar z)}}$ is determined by the equations
$$
\left((z_\C-z) \circ \eta_{-t}\right)(z_\C,w_\C)=0 \Leftrightarrow \left(e^{t\tau X_{h_\C}}\cdot z_\C\right)(z_\C,w_\C)=z,
$$
whose set of solutions intersects $M$ at the point $\varphi_{t\tau}^{-1}(z,\bar z)$, as we wanted.
\end{proof}

\begin{remark}Note that $\eta_t$ is the flow of $-X_{{\rm Re}\,\tilde h_\C}^{\frac12(\omega_\C+\bar\omega_\C)}$ 
due to a relative minus sign between our conventions and the conventions of \cite{BLU}. 
\end{remark}

\subsection{Real and mixed polarizations}
\label{ssnonholproj}

Let us consider a generalization of the geometric setting of \cite{BLU} motivated by our approach.
Assume that for some $t_0 \neq  0$, the structure $J_{t_0}$ of Theorem 1.2 of \cite{BLU} degenerates,
and the corresponding polarization of $(M,\omega)$ becomes mixed. To study this situation, let us consider diagram (\ref{diagram})
with complex structures replaced by polarizations (which are equivalent to complex structures in the K\"ahler case),
\be\label{diagram2}
\xymatrix{
&&M_\C    \ar[rr]^{\eta_t}&& M_\C \ar[d]^{\pi_t} \\
&&(M,{\mathcal P}_0) \ar @{^{(}->}[u]^{\iota}   && (M,{\mathcal P}_{t\tau}) \ar[ll]^{\varphi_{t\tau}} .}
\ee
Note that, in the case when $\varphi_{t\tau}$ is a diffeomorphism, one has 
$\varphi_{{t\tau}*}{\mathcal P}_{t\tau} = {\mathcal P}_0$. The polarization ${\mathcal P}_0$ can be described locally by the 
${\mathcal P}_0$-polarized ($ie \, J_0$-holomorphic) functions
$\{z^j_\alpha\}_{j=1,\dots,n}$. Then, ${\mathcal P}_{t\tau}$ will be described 
locally by the functions
\be\label{polari}
\{\varphi_{t\tau}^*z^j_\alpha\}_{j=1,\dots,n}. 
\ee
In the case when $\varphi_{t\tau}$ is well defined but no longer a diffeomorphism, (\ref{polari}) still defines a 
polarization on $M$. 
To illustrate this situation, let us consider an example. 
\begin{example} Consider again example \ref{planosymp}. As seen in example \ref{plano}, if $\tau=-is_0$
the linear map $\varphi_{-is_0}$ is not invertible, mapping $\R^2$ to the dimension one subspace $\{y=0\}$. 
The polarization ${{\mathcal P}_0}$ is  defined by the polarized ($ie$ $J_0$-holomorphic) function 
${z_0=x+\tau_0 y}$. The polarization
${\mathcal P}_{-is_0}$ is defined by the polarized function  $z_0\circ \varphi_{-i s_0}= x + r_0 y$, as in example \ref{planosymp}.

Note also that on the complexification $M_\C=\C^2$, for some $\tau\in \C$, the leaves of the foliation 
${\mathcal F}_t$ are given by the conditions $(z_{t\tau})_\C=x_\C+(\tau_0+t\tau)y_\C=const.$ When $t\tau=-is_0$ (or more generally 
${\rm Im}(t\tau) = -s_0$), these leaves either  do not intersect $M=\R^2 \subset \C^2$ or are contained in $\R^2$ so that the projection $\pi_t$ is not defined. 
\end{example}

We see that, for a mixed polarization ${\mathcal P}_{t\tau}$ in (\ref{diagram2}) the projection
$\pi_t$ is in general not defined. In fact, it will always be the case that if  ${\mathcal P}_{t\tau}$ 
is real then the leaves of the associated foliation of $M_\C$ (defined by the level sets of the analytic continuation of 
the ${\mathcal P}_{t\tau}$-polarized functions) either do not intersect $M$ or are contained in $M$. 
On the other hand, since the projection $\pi_0$ is 
well defined, by assumption, by inverting the vertical and top horizontal arrows in diagram (\ref{diagram2})
Proposition \ref{blue} can be extended in the following sense. 
\begin{proposition}
Assume that $\eta_t$ is defined for some $t\in \R$. The composition $\pi_0 \circ \eta_{-t} \circ \iota$ defines the map 
$\varphi_{t\tau}:M\to M$, as in the diagram 
\be\label{diagram3}
\xymatrix{
&&M_\C   \ar[d]_{\pi_0} && M_\C \ar[ll]_{\eta_{-t}} \\
&&(M,{\mathcal P}_0)   && (M,{\mathcal P}_{t\tau}) \ar[ll]^{\varphi_{t\tau}} \ar @{^{(}->}[u]_{\iota} .}
\ee
Let $\{(U_\alpha,z_\alpha)\}$ be a set of local holomorphic charts for $(M,J_0)$, covering $\pi_0\circ \eta_{-t}\circ \iota (M)$. 
Then, 
the local functions $\{\varphi_{t\tau}^*(z_{\alpha}^j)\}_{j=1,\dots,n}$ define a polarization ${\mathcal P}_{t\tau}$ on $M$.
\end{proposition}

\begin{proof}
The functions $\{\varphi_{t\tau}^*(z_{\alpha}^j)\}_{j=1,\dots,n}$ are defined locally on $M$ and for each $p\in M$ there is at least one index $\alpha$ for which the functions $\varphi_{t\tau}^*(z_{\alpha}^j), j=1,\dots,n$ are well defined in a 
open neighborhood of $p$. The analytic continuation of $z_\alpha^j$ from $U_\alpha$ to $\pi_0^{-1}(U_\alpha)$ is $(z_\alpha^j)_\C$.
The local funtions $\{(z_\alpha^j)_\C\}_{j=1,\dots,n}$ on $M_\C$ are functionally independent and define a local Lagrangian (with respect to $\omega_\C$)
foliation of $M_\C$, whose leaves coincide with the leaves of ${\mathcal F}_0$. Under pull-back by 
the Hamiltonian flow $\eta_{-t}$, which by assumption is well defined and is therefore a diffeomorphism, these functions 
remain functionally independent and continue to define a local Lagrangian distribution. Since $M$ is a totally 
real Lagrangian submanifold of $M_\C$, the restriction to $M$ by $\iota$ still gives a set of functionally independent functions 
defining, locally, a polarization on $M$. The fact that these locally defined polarizations glue consistently to give a 
globally defined polarization ${\mathcal P}_{t\tau}$ follows from the fact that if, locally, 
$(z_\alpha)_\C=\phi_{\alpha\beta}((z_\beta)_\C$, with $\phi_{\alpha\beta}$ a local biholomorphism, then $\eta_{-t}^*(dz_\alpha^j)_\C$ is a linear combination of 
$\{\eta_{-t}^*(dz_\beta^j)_\C\}_{j=1,\dots,n}$.
\end{proof}

\subsection{Non real analytic cases}
\label{ssnonanal}

The approach described in Theorems \ref{ajax} and \ref{kahlerfamily} can still be useful in situations where the Hamiltonian 
$h$ is not real analytic.
In this case, one cannot use a flow on $M_\C$ to describe the evolution of 
K\"ahler structures on $M$ but the local complex time Hamiltonian action on local coordinates can sometimes still be considered. 

The prototypal non real analytic example is given by $(M,\omega,J_0)=(\R^{2n},\sum_{j=1}^n dx^j\wedge dy^j, J_0)$, where $J_0$ 
is the standard complex structure on $\R^{2n}\cong \C^n$, when one considers a smooth (non real analytic) function 
Hamiltonian $h$ depending only on the variables $(y^1,\dots,y^n)$. One has,
$$
X_h = \sum_{j=1}^n \frac{\partial h}{\partial y^j} \frac{\partial}{\partial x^j}.
$$ 
The flow of $X_h$ is given by $\gamma_t(x,y)= (x+t \frac{\partial h}{\partial y}, y)$, where $\frac{\partial h}{\partial y}=
(\frac{\partial h}{\partial y^1}, \dots, \frac{\partial h}{\partial y^j})$.
The Lie series corresponding to the global $J_0$--holomorphic coordinates, $z=x+iy$, has only the 
first two terms different from zero
and thus defines,  for any $\tau = r+is \in \C, r,s\in\R$, new global functions defining 
the new polarization, ${\cal P}_\tau$,
$$
z_\tau = e^{\tau X_h} (z) = x + r \frac{\partial h}{\partial y} + i  (y + s  \frac{\partial h}{\partial y}).
$$
The map $\varphi_\tau \, : \,  (\R^{2n}, {\cal P}_\tau)  \rightarrow (\R^{2n}, {\cal P}_0) $
{\em ie} such that
$$
z_\tau = (\varphi_\tau)^* z,
$$
is, gobally in $\R^{2n} \times \C$, given by the following map, which is analytic in $x$ but not in $y$,
$$
\varphi_\tau(x, y) =(x+r  \frac{\partial h}{\partial y}, y + s  \frac{\partial h}{\partial y}).
$$ 
If $h$ has a nonnegative Hessian then we see that  $\varphi_\tau$ are diffeomorphisms
for $s \geq 0$. 
Therefore, in this situation, Theorem \ref{ajax} is globally valid.

A more elaborate example where one can treat non real analytic Hamiltonian flows is given by the cotangent bundle of a 
compact Lie group, $T^*K$, where $h$ is a smooth $K$-bi-invariant function. This case was treated in \cite{KMN1} and it is 
further described in Section \ref{ss33}.

\subsection{Completing flows}
\label{sscompleting}

The geometric setting of \cite{BLU} makes it possible, in some cases,
to extend our map $\varphi_\tau$ in Theorem \ref{ajax}.
Important examples are given by real algebraic completely 
integrable systems (aci) \cite{Va}, $(M, \omega, \mu)$, where $\mu : M \rightarrow \R^n$,
is the moment map of a $\R^n$ action such that the 
complexification $(M_\C, \omega_\C, \mu_\C)$ is a nonsingular affine 
variety and the regular level sets of $\mu_\C$ are affine parts 
of abelian varieties where the flows of $X_{\mu^j_\C}$ linearize. Therefore,
there is an extension  $(\widehat M, \widehat \omega, \widehat \mu)$ of the original system,
with $M$ dense in $\widehat M$,  
in which  $\varphi_\tau$ can be geometrically extended to $\tau \in \C$.

\section{Geodesics on the space of K\"ahler metrics}
\label{sectiong}

Let us consider the families of complexified Hamiltonian flows of 
the previous section. As remarked in \cite{Do1}, geodesics 
in the space of K\"ahler metrics, which is viewed as a symmetric space for the group of 
symplectomorphisms, can be thought of 
as analytic continuations of Hamiltonian flows which are the one-parameter subgroups of the group of symplectomorphisms. 
This can be made explicit as follows.

\begin{proposition}\label{geod} Under the conditions of Theorem 
\ref{kahlerfamily}, let $\tau = it, t\in\R$. Then, the path of K\"ahler metrics $\gamma_\tau$, in $(i)$ in the 
Theorem, is a geodesic path.
\end{proposition}

\begin{proof}
Let $\varphi_\tau$ be the diffeomorphism defined in Theorem \ref{ajax}, for $\tau=it, t\in \R,$ so that, in 
local complex coordinates with respect to $J_0$, $$\varphi_\tau(z,\bar z) = (e^{itX_h} z, e^{-itX_h} \bar z).$$ 
Letting $\omega_\tau=(\varphi_\tau^{-1})^* \omega$, note that, from the proof of Theorem \ref{kahlerfamily}, 
equations (\ref{moser}), (\ref{phitokappa}) and from $J_\tau = \varphi_\tau^*J_0$, it follows 
that $\omega_\tau$ is of type $(1,1)$ with respect to $J_0$. In this way, we have the Moser maps
$$
\varPhi_t=\varphi_\tau:(M,J_\tau,\omega)\to (M,J_0,\omega_\tau), \, \tau = it, t\in \R.
$$

Recall that, 
$\kappa_\tau = -2 {\rm Im}\,\psi_\tau$, where 
$$
\psi_\tau = -\frac{i}{2}\, e^{\tau X_h}\cdot \kappa_0 + \tau h - \alpha_\tau,  
$$
and $d\alpha_\tau= e^{\tau d\circ {i}_{X_h}}\theta -\theta$, with $\omega = -d\theta$. One has,
$$
\dot \kappa_\tau := \frac{d\kappa_\tau}{dt}= -X_h(\psi_\tau+\bar\psi_\tau) - 2h + 2\theta(X_h),
$$
where we used 
$$
\dot \alpha_\tau = iX_h(\alpha_\tau) + i \theta(X_h).
$$
Using (\ref{phitokappa}), let $\phi_t = \kappa_\tau \circ \varPhi^{-1}_t -\kappa_0$, so that
$$
\phi_t (z,\bar z) = \kappa_\tau (e^{-\tau X_h}z, e^{\tau X_h \bar z})-\kappa_0(z,\bar z), \, \tau = it, t\in \R. 
$$
Then, 
$$
\dot \phi_t = \dot \kappa_\tau + \partial_\tau\kappa_\tau(-iX_h)+\bar\partial_\tau\kappa_\tau (iX_h).
$$
Using $d\psi_\tau (X_h)=X_h(\psi_\tau)=\partial_\tau \psi_\tau(X_h)+\bar\partial_\tau \psi_\tau(X_h)$ and $\theta^{(0,1)_\tau}=
\bar\partial_\tau \psi_\tau$, we obtain finally
\begin{equation}\label{velocity}
\dot \phi_t = -2h \circ \varPhi_t^{-1}.
\end{equation}
It follows that, for $\tau =it, t\in \R$,
$$
\ddot \phi_t \circ \varPhi_t = 2i \partial_\tau h (X_h) - 2i \bar\partial_\tau h (X_h) 
=2dh (J_\tau X_h) = 2 \omega(X_h,J_\tau X_h) = 2 ||X_h||^2_{\gamma_\tau}.
$$
Since we are on a K\"ahler manifold we have 
$$
||X_h||^2_{\gamma_\tau} = ||dh||^2_{\gamma_\tau}.
$$
On the other hand, since $\varPhi_t^*  \tilde \gamma_\tau = \gamma_\tau$, for $\tau=it, t\in \R$, we have
$$
\nabla \dot\phi_t = -2 \nabla h \circ   \varPhi_t^{-1},
$$ 
where on the left hand side we have the gradient of $\dot\phi_t$ with respect to the metric $\tilde \gamma_\tau$, while 
on the right hand side we have the gradient of $-2h$ with respect to the metric $\gamma_\tau$.
Therefore, using $||\nabla h||_{\gamma_\tau}^2=||dh||_{\gamma_\tau}^2$, we obtain the geodesic equation
$$
\ddot \phi_t = \frac12 ||\nabla \dot\phi_t||^2_{\tilde \gamma_\tau}, \, \tau = it, t\in \R,
$$
as we wanted. 
\end{proof}

\begin{remark}The minus sign in the expression for the velocity $\dot\phi_t$ in (\ref{velocity}) can be traced back to Theorem 
\ref{ajax}, where we see that the Moser map is $\varphi_\tau$ and not $\varphi_\tau^{-1}$.
\end{remark}

\begin{remark}Note that in Proposition \ref{geod} the simplification of taking $\tau$ to be pure imaginary gives 
no essential loss of generality. For real time $\tau=s\in \R$, the diffeomorphisms $\varphi_s$ form an actual Hamiltonian flow, 
with $\omega_s=\omega$. 
In this case, as a consequence of Proposition \ref{realnada}, 
$\kappa_s = \kappa_0\circ \varphi_s$ and $\phi_s=0$ so that we always obtain diffeomorphic K\"ahler structures
$(M,\omega,J_0,\gamma_0) \cong (M,\omega,J_s,\gamma_s)$, where $J_s=\varphi_s^*J_0, \gamma_s = \varphi_s^*\gamma_0$. So, 
to actually move in the ``moduli space of K\"ahler structures modulo diffeomorphism'' we need a non-zero imaginary part of 
$\tau$.
\end{remark}

\begin{remark}The geodesic imaginary time families of K\"ahler structures considered in Theorem \ref{kahlerfamily} 
and in Proposition \ref{geod} have been studied for symplectic toric manifolds \cite{BFMN, KMN1} and 
for cotangent bundles of Lie groups of compact type \cite{KMN1,KMN2}. 
It is straightforward to check that the K\"ahler potentials  presented in those works are instances of formula 
(\ref{formula}). Cotangent bundles of compact Lie groups are also described in Section \ref{ss33}. 
In these cases, the diffeomorphisms $\varphi_\tau$ exist for large values of $\tau$. Further 
investigation of the toric case will appear in \cite{KMN3}.
\end{remark}

\section{Cotangent bundle of a compact Lie group}

\label{ss33}
In this section, for completeness, we recall some of the results of \cite{KMN1} and show that the 
families of bi-invariant K\"ahler structures on cotangent bundles of compact Lie groups described 
there are geodesic families. (The case of the quadratic symplectic potential has also been studied in \cite{HK2}.)
This is an example where one can analytically continue to complex time Hamiltonian flows 
of not necessarily real analytic Hamiltonian functions. (See Section \ref{ssnonanal}.)

Let $K$ be a compact Lie group. Infinite families of $K\times K$-invariant K\"ahler 
structures on $TK\cong K\times Lie(K) \cong K_\C$, where $K_\C$ is the complexification of $K$, 
were studied in \cite{KMN1}. 

Let $h\in C^{\infty}(Lie(K))$ be an Ad-invariant strictly convex function such that the inverse of the 
Hessian of $h$ has bounded operator norm. 
Let $\{y^i\}_{i=1,\dots, n}$ be cartesian coordinates on 
$Lie(K)$ associated to a choice of orthonormal baisis with respect to the Killing metric. Let
\begin{equation}\label{legendre}
u^i=\frac{\partial h}{\partial y^i},\,\, i=1,\dots,n.
\end{equation}

Then, the map 
\begin{eqnarray}\nonumber
T^*K\cong K\times Lie(K) &\stackrel{\psi_0}{\to}& K_\C\\ \nonumber
(x,Y) &\mapsto& xe^{i u(Y)},
\end{eqnarray}
is a diffeomorphism and defines on $T^*K$ (and on $K_\C$) a K\"ahler structure $(T^*K, \omega, J_0)$ \cite{KMN1}.
 To study
the action of the Lie series $e^{\tau X_h}$ on the polarization ${\cal P}_0$ corresponding
to $J_0$
it is sufficient to study its action on $J_0$--holomorphic functions of the form,
$f_{\pi, E_\pi}(x, Y) = tr(E_\pi \pi(x e^{iu(Y)})$, where $\pi$ denotes a  finite dimensional representation of $K_\C$
and $E_\pi$  an endomorphism of the space of the representation. 
\begin{proposition} \label{assi}
\be
\label{ass}
e^{\tau X_h}(f_{\pi, E_\pi})(x, Y) =  tr(E_\pi \pi(x e^{(i+\tau)u(Y)}))
\ee
\end{proposition}

\begin{proof}
This is a consequence of Theorem 3.7 in \cite{KMN1}. There, it is shown that the operator $e^{\tau' X_h}$ 
applied to $tr(E_\pi \pi(x))\in C^\infty(K)$ gives
$$
e^{\tau' X_h} \cdot tr(E_\pi \pi(x)) = tr(E_\pi \pi (xe^{\tau' u(Y)})).
$$
On the other hand, for $\tau'=\tau+i$,
$$
e^{(\tau+i) X_h} \cdot tr(E_\pi \pi(x)) =e^{\tau X_h} \cdot tr(E_\pi \pi(xe^{iu(Y)})), 
$$
which proves the proposition. 
For a more explicit proof, note that, from \cite{KMN1},
\be\label{hamvectstar}
X_h = \sum_{i=1}^n u^i(Y) X_i,
\ee
where $\{X_i\}_{i=1,\dots,n}$ is a frame of left-invariant vector fields on $K\times Lie(K)$, with zero 
component along $Lie(K)$. Therefore, we have
$$
X_h \cdot \pi(xe^{iu(Y)}) = \pi(x) \pi(u(Y)) \pi(e^{iu(Y)}), 
$$
where we also denote by $\pi$ the representation induced by $\pi$ on $Lie(K)$. Therefore,
$$
e^{\tau X_h} \cdot tr(E_\pi \pi(xe^{iu(Y)})) = tr(E_\pi \pi(x) \pi(e^{\tau u(Y)}) \pi(e^{iu(Y)})) = 
tr(E_\pi \pi (xe^{(\tau+i) u(Y)})).
$$
\end{proof}

\begin{remark}
Notice that Proposition \ref{assi} does not require the smooth function $h$ to be real analytic.
\end{remark}

\noindent The functions (\ref{ass}) define a polarization ${\cal P}_\tau$, which is K\"ahler for 
$s = {\rm Im}\,(\tau) >-1$ (see Theorem \ref{thmkmn1}),  real 
for $s=-1$ and  pseudo-K\"ahler  for $s<-1$.
The map $\phi_\tau \, : T^*K \rightarrow T^*K$, mapping $J_0$ holomorphic functions
to ${\cal P}_\tau$ polarized functions, is easily seen to be given by
$$
\varphi_\tau = \psi_0^{-1} \circ \psi_\tau,
$$
where 
\bas
T^*K\cong K\times Lie(K) &\stackrel{\psi_\tau}{\to}& K_\C\\ \nonumber
(x,Y) &\mapsto& xe^{(i+\tau) u(Y)}
\eas

\begin{theorem}\label{thmkmn1}(\cite{KMN1}) For $s={\rm Im}\,\tau >-1$,
$(T^*K,\omega,\varphi_\tau^*J_0)$ is a K\"ahler manifold, with $K\times K$ invariant K\"ahler structure and 
K\"ahler potential obtained by a Legendre transform of the $K\times K$-invariant Hamiltonian $h$,
$$
\kappa_\tau (u(Y)) = 2(s+1) (Y\cdot u(Y)-h(Y)).
$$   
\end{theorem}

It is straightforward to check that the K\"ahler potentials $\kappa_\tau$ correspond to a 
geodesic in the space of K\"ahler metrics generated by the analytic continuation to imaginary time of the flow 
of ${X_h}$. In fact,

\begin{proposition}
The K\"ahler potential $\kappa_\tau$ in Theorem \ref{thmkmn1} satisfies equation (\ref{formula}) in Theorem \ref{kahlerfamily}.
\end{proposition}

\begin{proof}A symplectic potential is $\theta = \sum_{i=1}^ny^iw_i$, where 
$\{w_i\}_{i=1,\dots,n}$ is the frame of left-invariant 1-forms on $K$ dual to  $\{X_i\}_{i=1,\dots,n}$ and 
pulled-back to $T^*K$ by the canonical 
projection. One then has, from (\ref{hamvectstar}), $\theta(X_h)= u(Y)\cdot Y$ and $\iota_{X_h}d(\theta(X_h))=0$, so that 
$\alpha_\tau$ in Theorem \ref{kahlerfamily} is
$$
\alpha_\tau = \tau  u(Y)\cdot Y.
$$

Then, according to Theorem \ref{kahlerfamily},
$$
\kappa_\tau = -2{\rm Im}\, \left(-\frac{i}{2}e^{\tau X_h}\kappa_0 +\tau h - \tau u(Y)\cdot Y\right).
$$
Since (see \cite{KMN1}) $\kappa_0=2(Y\cdot u(Y)-h(Y))$ in this case, from (\ref{hamvectstar}) we have $X_h(\kappa_0)=0$ and 
$\kappa_\tau = 2({\rm Im}\,\tau+1) (u(Y)\cdot Y -h)$,
as required.
\end{proof}

\bigskip
\large{\bf Acknowledgements:} We would like to thank the referee for several corrections and useful suggestions.
We wish to thank Will Kirwin for many discussions and Brian Hall for corrections 
to the first version of the paper. This work was partially supported by CAMGSD-LARSyS through the FCT 
Program POCTI-FEDER and by the FCT projects PTDC/MAT/119689/2010 and EXCL/MAT-GEO/0222/2012.

\end{document}